\documentclass[12pt,a4paper,twoside,english,american]{scrartcl}
\usepackage{url}
\usepackage{helvet}
\usepackage[T1]{fontenc}
\usepackage[latin9]{inputenc}
\usepackage{amsmath}
\usepackage{setspace}
\usepackage{amssymb}
\usepackage{enumerate}
\usepackage{mathtools}

\makeatletter

 \usepackage{amsthm}
 \theoremstyle{plain}
\newtheorem{thm}{Theorem}[section]
  \theoremstyle{plain}
  \newtheorem*{thm*}{Theorem}
  \theoremstyle{plain}
  
  \theoremstyle{remark}
  \newtheorem{rem}[thm]{Remark}
  \theoremstyle{definition}
  
 \theoremstyle{definition}
 \newtheorem*{defn*}{Definition}
  \theoremstyle{plain}
  
 \theoremstyle{definition}
  
  \theoremstyle{remark}

\usepackage[T1]{fontenc}    

\usepackage{a4wide}        
\usepackage{tu-preprint}
\usepackage{amsmath}

\usepackage{euscript}
\allowdisplaybreaks[1] 
\usepackage{amsfonts}
\newenvironment{keywords}{ \noindent\footnotesize\textbf{Keywords and phrases:}}{}

\newenvironment{class}{\noindent\footnotesize\textbf{Mathematics subject classification 2010:}}{}

\usepackage{color,tu-preprint}

\newcommand{\Abs}[1]{\left\lVert#1\right\rVert}
\newcommand*{\trace}{\operatorname{trace}}

\newcommand*{\curl}{\operatorname{curl}}
\newcommand{\R}{\mathbb{R}}
\newcommand*{\Grad}{\operatorname{Grad}}

\newcommand*{\abs}[1]{\lvert#1\rvert}

\newcommand*{\grad}{\operatorname{grad}}

\DeclareMathAccent{\Circ}{\mathalpha}{operators}{"17}

\renewcommand{\Re}{\operatorname{\mathfrak{Re}}}

\theoremstyle{plain}
\newtheorem{Sa}[subsubsection]{Theorem}
\newtheorem{Sa*}[subsection]{Theorem}
\newtheorem{Le}[subsubsection]{Lemma}
\newtheorem{Le*}[section]{Lemma}
\newtheorem{Fo}[subsubsection]{Corollary}
\newtheorem{Fo*}[subsection]{Corollary}

\newtheorem{Prop*}[section]{Proposition}

\theoremstyle{definition}
\newtheorem*{Def}{Definition}

\theoremstyle{remark}

\newtheorem{rems}[subsubsection]{Remarks}

 \numberwithin{equation}{section}

\DeclareMathOperator{\TextRe}{Re}

\renewcommand{\Re}{\TextRe}

\DeclareMathOperator{\Lin}{Lin}

\newcommand{\N}{\mathbb{N}}

\newcommand{\eps}{\varepsilon}

\newcommand{\dd}{\ \mathrm{d}}

\DeclareMathOperator{\Diverg}{Div}
\DeclareMathOperator{\diverg}{div}

\newcommand{\ben}{\begin{enumerate}[(i)]}
\newcommand{\een}{\end{enumerate}}

\renewcommand{\tilde}{\widetilde}
\renewcommand*{\epsilon}{\varepsilon}
\renewcommand*{\rho}{\varrho}

\makeatother

\usepackage{babel}

\begin{document}
\selectlanguage{english}%
\institut{Institut f\"ur Analysis}

\preprintnumber{MATH-AN-03-2011}

\preprinttitle{A note on elliptic type boundary value problems with maximal monotone relations.}

\author{Sascha Trostorff, Marcus Waurick}

\makepreprinttitlepage

\selectlanguage{american}%


\title{A note on elliptic type boundary value problems with maximal monotone relations.}

\author{Sascha Trostorff \\ Marcus Waurick\\
 Institut f\"ur Analysis, Fachrichtung Mathematik\\
 Technische Universit\"at Dresden\\
 Germany\\
sascha.trostorff@tu-dresden.de\\
 marcus.waurick@tu-dresden.de }
\maketitle
\begin{abstract}
In this note we discuss an abstract framework
for standard boundary value problems in divergence form with maximal
monotone relations as {}``coefficients''. A reformulation of the respective
problems is constructed such that they turn out to be unitary equivalent
to inverting a maximal monotone relation in a Hilbert space. The method
is based on the idea of {}``tailor-made'' distributions
as provided by the construction of extrapolation spaces, see e.g.
[Picard, McGhee: \emph{Partial Differential Equations: A unified
Hilbert Space Approach}. DeGruyter, 2011]. The abstract
framework is illustrated by various examples.
\end{abstract}

\begin{keywords} Elliptic Differential Equations in Divergence form; maximal monotone relations; Gelfand triples \end{keywords}

\begin{class} 35J60, 35J15, 47H04, 47H05\end{class}
\newpage
\tableofcontents{}

\newpage
\section{Introduction}

In mathematical physics elliptic type problems play an important
role, in analyzing various equilibria as for example in potential
theory, in stationary elasticity and many other types of stationary
boundary value problems. Classical monographs, focusing
mainly on linear problems, are for instance \cite{agmon2010lectures,Courant1962,Evans2010,gilbarg2001elliptic}.
We also refer to \cite[Chapter VIII]{Showalter1994} for a survey
of the literature. Also non-linear elliptic type problems have been
studied intensively. The authors of \cite{Amann1979,Amann1998} study
non-linear perturbations of a selfadjoint operator and obtain existence of a solution. Later-on,
uniqueness results could be proved, see \cite{Amann1972,Serrin1972}.
Operators in divergence form with non-linear coefficients are studied
in \cite{Asakawa1987,Browder1963,Browder1965,Browder1966,Visik1963},
where some monotonicity condition is imposed on the coefficients to
obtain existence. This monotonicity condition might also be very weak,
cf. \cite{Hungerbuhler1999}. The case of divergence form operators
with multi-valued coefficients is treated among other things in \cite{CDC},
where also existence results could be obtained. In this note, we restrict
ourselves to the Hilbert space setting and study conditions under
which abstract divergence form operators with possibly multi-valued
coefficients lead to well-posed operator inclusions. The restriction
to the Hilbert space case enables us to show continuity estimates
also for inhomogeneous boundary value problems of elliptic type, cf.
the Corollaries \ref{cor: cont_estimate} and \ref{cor: cont_est_vn},
where -- to the best of the authors' knowledge -- the first
one is new. The main topic is the discussion of the structure of
the following type of problem: Let $H_{1}$, $H_{2}$ and $G_{1}$,
$G_{2}$ be Hilbert spaces and let $f\in G_{1}$ be given. Moreover,
let $a\subseteq H_{2}\oplus H_{2}$ be a relation such that $a^{-1}:H_2\to H_2$ becomes a Lipschitz-continuous mapping (the main focus will be laid on $c$-maximal monotone relations, which will be defined below),
$A:D(A)\subseteq H_{1}\to H_{2}$ densely defined closed linear. We
study the problem of finding $u\in G_{2}\subseteq D(A)$ such that
the inclusion 
\begin{equation}
\tag{\ensuremath{*}}A^{*}aA \ni (u,f),\label{inc:abstract}
\end{equation}
 holds true, i.e. there exists $w\in H_2$ such that
\begin{equation*}
 (Au,w)\in a \mbox { and } A^\ast w=f.
\end{equation*}
We want to find the {}``largest'' space $G_{1}$ to
allow for existence results and the {}``largest'' space $G_{2}$
to yield uniqueness. Endowing $G_{1}$ and $G_{2}$ with suitable
topologies, we seek a solution theory for these type of inclusions.
We will give a framework in order to cover inhomogeneous boundary
value problems with Dirichlet or Neumann boundary data. Compatibility
conditions such as in \cite[Theorem 4.22]{CioDon} arise naturally
in our approach, cf. also Remark \ref{rem:weak_solution}. 


%

Our approach consists in rewriting \eqref{inc:abstract} as an inclusion in ``tailor-made''
distributions spaces by introducing suitable extrapolation
spaces, which are also known as Sobolev chains or Sobolev
towers, see e.g. \cite{EngNag,PicSoLa} and the references therein.
The core idea is to generalize extrapolation spaces to the non-selfadjoint
operator case. This was also done and extensively used in \cite{Picard}
for studying time-dependent problems. Using this extrapolation spaces the abstract problem (\ref{inc:abstract}) turns out to be unitary equivalent to the problem of inverting the relation $a$ in a suitable space. 
Since elliptic type problems are not well-posed in general, one has
to develop a suitable framework in order to determine possible right-hand
sides. We discuss some preliminary facts in Section \ref{PSC} used
in Section \ref{AT}, which are particularly needed for the Theorems
\ref{Thm: STHE}, \ref{pro:inhom_Dirichlet}, \ref{pro:inhom_Neumann}
and the Corollaries \ref{cor: cont_estimate}, \ref{cor: cont_est_vn}.
These theorems and corollaries are the main results of this paper.
We discuss extrapolation spaces in Section \ref{OTF}. Section \ref{MMR}
contains some results in the theory of maximal monotone relations. Most importantly, the following problem is discussed: When is a composition of a orthogonal
projection with a maximal monotone relation again maximal monotone?
This question was also addressed in \cite{Asakawa1987,Garcia2006,Pennanen2003,Robinson1999}.
Particularly in \cite{Robinson1999}, this question was, at least for
our purposes, satisfactorily answered. For easy reference, we also
state some well-known results in the theory of maximal monotone relations
in Section \ref{MMR}. In Section \ref{AT} we apply the results of
the previous ones to give an abstract solution theory for both homogeneous
and inhomogeneous boundary value problems of elliptic type. In Section \ref{ex} we will give some examples, how the  abstract theory could be employed to study boundary value problems in potential theory, stationary elasticity and magneto- and electro-statics.

The underlying scalar field of any vector space discussed here is
the field of complex numbers and the scalar product of any Hilbert
space in this paper is anti-linear in the first component.

\section{Functional analytic preliminaries}\label{PSC}

\subsection{Operator-theoretic framework}\label{OTF}
We recall some definitions from operator theory. As general references we refer to \cite{Kato1976,Picard}.
\begin{Def}[modulus of $A$, cf. {\cite[VI 2.7]{Kato1976}}] Let $H_1,H_2$ be Hilbert spaces. Let $A: D(A)\subseteq 
H_1\to H_2$ be a densely defined closed linear operator. The operator $A^*A$ is non-negative and selfadjoint in
$H_1$. We define $\abs{A}:=\sqrt{A^*A}$ the \emph{modulus of $A$}. It holds $D(\abs{A})= D(A)$.
\end{Def}

The following notion of extrapolation spaces and extrapolation operators can be found in \cite{PicSoLa,Picard}. See in particular \cite{PicSoLa}, where a historical background is provided.
\begin{Def}[extrapolation spaces, Sobolev chain] Let $H$ be a Hilbert space. Let $C
: D(C)\subseteq  H \to H$ be a densely defined closed linear
operator and such that $0$ is contained in the resolvent set of $C$.
Define $H_1(C)$ to be the Hilbert space $D(C)$ endowed with the norm
$\abs{C\cdot}_H$. Define $H_0(C):= H$ and let $H_{-1}(C)$ be the
completion of $H_0(C)$ with respect to the norm
$\abs{C^{-1}\cdot}_H$. The triple $(H_{1}(C),H_0(C),H_{-1}(C))$ is called \emph{(short) Sobolev chain}.
\end{Def}

\begin{rems}\label{rem:ExDua}\begin{enumerate}[(a)]
  \item\label{rem: Extra_op} It can be shown that $C: H_1(C)\to H_0(C)$ is unitary. Moreover, the operator $C:H_1(C)\subseteq  H_0(C)\to H_{-1}(C)$ has a unique unitary extension, cf. \cite[Theorem 2.1.6]{Picard}.
  \item\label{rem:duality} Sometimes it is useful to identify $H_{-1}(C)$ with $H_{1}(C^\ast)'$, the dual space of $H_{1}(C^\ast)$ (cf. \cite[Theorem 2.2.8]{Picard}). This can be done by the following unitary mapping
\begin{align*}
 V: H_{1}(C^\ast)' &\to H_{-1}(C) \\
    \psi &\mapsto C R_H (H \ni u \mapsto \psi((C^\ast)^{-1}u)),
\end{align*}
where $R_H: H'\to H$ denotes the Riesz-mapping of $H$. Its inverse is given by
\begin{align*}
 V^\ast :H_{-1}(C) &\to  H_{1}(C^\ast)' \\
      u &\mapsto (H_1(C^\ast)\ni v \mapsto \langle C^{-1}u , C^\ast v\rangle_H ).
\end{align*}
By this unitary mapping we can identify $Cx\in H_{-1}(C)$ for $x\in H$ with the functional
\begin{equation*}
 \langle Cx,y\rangle_{H_{-1}(C)\times H_1(C^\ast)}=\langle x,C^\ast y\rangle_H.
\end{equation*}
\end{enumerate}
\end{rems}

We apply the above to the following particular situation. It should be noted that at least for selfadjoint operators a similar strategy has been presented in \cite{Amann1979}. Let $H_1, H_2$ be Hilbert spaces and let $A: D(A)\subseteq  H_1 \to H_2$ be a densely defined closed linear operator such that the range of $A$, $R(A)$, is closed in $H_2$. Recall that $R(A)=N(A^*)^\bot$ and $\overline{R(A^*)}=N(A)^\bot$. The main idea of formulating elliptic type problems is to use the Sobolev chain of the modulus of 
\[
   B : D(A)\cap N(A)^\bot \subseteq  N(A)^\bot \to R(A):\phi\mapsto A\phi
\]
and the modulus of the respective adjoint. 

\begin{Le}\label{Le: Adjoint} The following holds
\[
   B^* : D(A^*)\cap N(A^*)^\bot\subseteq  N(A^*)^\bot \to
   N(A)^{\bot}: \phi\mapsto A^*\phi.
\]
\end{Le}
\begin{proof} Let $(u,v)\in R(A)\oplus N(A)^\bot$. Then we have\footnote{Occasionally, we will identify an operator $B$ with its graph, i.e., $B = \{(x,Bx); x\in D(B)\}$.} 
\begin{align*}
   (u,v)\in B^* &\iff \forall \phi \in D(B): \langle B\phi, u\rangle
   = \langle \phi , v \rangle \\
   & \iff \forall \phi \in D(A)\cap N(A)^\bot : \langle A\phi,u
   \rangle = \langle \phi, v \rangle \\
   & \iff \forall \phi \in D(A) : \langle A\phi,u
   \rangle = \langle \phi, v \rangle \\
   & \iff (u,v)\in A^*.\qedhere	
\end{align*}
\end{proof}

We note that since $R(A)$ is closed, the operator $B^{-1}$ is continuous by the closed graph theorem. We may show a similar property for $B^*$.

\begin{Fo} It holds $(B^*)^{-1} = (B^{-1})^*$ and
$\overline{R(A^*)}=R(A^*)$.
\end{Fo}
\begin{proof} The first equality is clear. Moreover, we deduce that
$(B^*)^{-1}$ is continuous and closed. Hence, $\overline{R(A^*)}=N(A)^\bot= D((B^*)^{-1})= R(B^*) = R(A^*)$.
\end{proof}

\begin{Sa}\label{Le: unitary} The operators $\abs{B}$ and $\abs{B^*}$ are continuously
invertible. Moreover, the operator
\begin{equation*}
   B: H_1(\abs{B}) \to H_0(\abs{B^*})
\end{equation*}
is unitary and the operator
\begin{equation*}
   B^\ast : H_1(\abs{B^*})\subseteq H_0(\abs{B^*}) \to H_{-1}(\abs{B})
\end{equation*}
can be extended to a unitary operator from $H_0(\abs{B^\ast})$ to $H_{-1}(\abs{B})$.
\end{Sa}
\begin{proof}
 As $B$ and $B^*$ are continuously invertible, so is $B^*B$. Thus, the spectral theorem for self-adjoint operators implies the continuous invertibility of $\abs{B}$. Interchanging the roles of
 $B$ and $B^*$, we get the continuous invertibility of $\abs{B^*}$.
 Now, let $\phi\in H_1(\abs{B})$. Then we have
 \[
    \abs{B\phi}_{H_0(\abs{B^*})} =
    \abs{\abs{B}\phi}_{H_0(\abs{B^*})}=\abs{\phi}_{H_1(\abs{B})}.
 \]
 Since $H_0(\abs{B^\ast})=R(A)$ the operator $B$ is clearly onto and hence unitary. Now, for $B^*$ it
 suffices to show that the norm is preserved for $\phi\in
 H_1(\abs{B^*})$. Let $\phi\in H_1(\abs{B^*})$. 
Using \cite[Lemma 2.1.16]{Picard} for the transmutation relation $\abs{B}^{-1}B^*\phi = B^*\abs{B^*}^{-1}\phi$, we conclude that
 \[
    \abs{B^*\phi}_{H_{-1}(\abs{B})}=
    \abs{\abs{B}^{-1}B^*\phi}_{H_{0}(\abs{B})}=
    \abs{B^*\abs{B^*}^{-1}\phi}_{H_{0}(\abs{B})}=
    \abs{\phi}_{H_0(\abs{B})}.\qedhere
 \]
\end{proof}
\begin{rem}{\label{rem:relation_Sobolev_chains}}
\begin{enumerate}[(a)]
\item We can construct the Sobolev chains of the operators $|A|+i$ and $|A^\ast|+i$, respectively. The operator $A$ can then be established as a bounded linear operator $A:H_k(|A|+i)\to H_{k-1}(|A^\ast|+i)$ for $k\in \{0,1\}$ (cf. \cite[Lemma 2.1.16]{Picard}). In virtue of Remark \ref{rem:ExDua}\eqref{rem:duality}, the element $Ax$ for $x\in H_0(|A|+i)$ can be interpreted as a bounded linear functional on $H_1(|A^\ast|+i)$. If $U$ denotes the partial isometry such that $A=U|A|$ (cf. \cite[VI 2.7, formula (2.23)]{Kato1976}), we compute for $y\in H_1(|A^\ast|+i)$
\begin{align*}
& \langle Ax,y\rangle_{H_{-1}(|A^\ast|+i),H_1(|A^\ast|-i)} \\ &= \langle (|A^\ast|+i)^{-1} Ax, (|A^\ast|-i)y\rangle_{H_2} \\
&=  \langle A (|A|+i)^{-1} x, (|A^\ast|-i)y\rangle_{H_2} \\
&= \langle U|A| (|A|+i)^{-1} x, (|A^\ast|-i)y\rangle_{H_2} \\
&= \langle Ux,(|A^\ast|-i)y\rangle_{H_2}+i\langle U(|A|+i)^{-1}x,(|A^\ast|-i)y\rangle_{H_2} \\
&= \langle x,U^\ast (|A^\ast|-i) y\rangle_{H_1}+i\langle (|A|+i)^{-1}x,(|A|-i)U^\ast y\rangle_{H_1} \\
&= \langle x, U^\ast (|A^\ast|-i)y+iU^\ast y\rangle_{H_1} \\
&= \langle x, A^\ast y\rangle_{H_1}.
\end{align*}
\item  We clearly have $H_1(|B|)=D(B)\subseteq D(A)=H_1(|A|+i)$ and $H_0(|B|)=R(A^\ast)\subseteq H_2=H_0(|A|+i)$. Since $H_{-1}(|B|)$ is defined as the completion of $R(A^\ast)$ with respect to the norm $||B|^{-1}\cdot|$ and since this norm is equivalent to the norm $|(|A|+i)^{-1}\cdot|$, we also get $H_{-1}(|B|)\subseteq H_{-1}(|A|+i)$. Clearly the analogue results hold for the Sobolev chains of $|B^\ast|$ and $|A^\ast|+i$.  
\end{enumerate}
\end{rem}

\subsection{Maximal monotone relations}\label{MMR}

We begin to introduce some notions for the treatment of relations.

\begin{Def}
 For a binary relation $a\subseteq H\oplus H$ and an arbitrary subset $M\subseteq H$ we denote by 
\begin{equation*}
 a[M]\coloneqq \{y\in H\, ; \, \exists x\in M: (x,y)\in a\}
\end{equation*}
the \emph{post-set of} $M$ \emph{ under } $a$ and by 
\begin{equation*}
 [M]a\coloneqq \{x\in H\,;\,\exists y\in M: (x,y)\in a\}
\end{equation*}
the \emph{pre-set of $M$ under $a$}.\\
The relation $a$ is called \emph{monotone} if for all pairs $(u,v),(x,y)\in a$ the following holds
\begin{equation*}
 \Re \langle u-x , v-y\rangle \geq 0,
\end{equation*}
and \emph{maximal monotone}, if for ever monotone relation $b$ with $a\subseteq b$ it follows that $a=b$. \\
Finally we define for a constant $c\in \mathbb{C}$ the relation $a-c\subseteq H\oplus H$ by
\begin{equation*}
 a-c\coloneqq \{(u,v)\in H\oplus H \,;\, (u,v+cu)\in a\}
\end{equation*}
and $a$ is called \emph{$c$-maximal monotone} if $a-c$ is maximal monotone.
\end{Def}

A reason for the treatment of maximal monotone relations as natural generalization of positive semi-definite linear operators is the following theorem. 
\begin{Sa}[{\cite[Theorem 1.3]{Morocsanu1988}}]\label{Thm: Moro} Let $a\subseteq  H\oplus H$ be monotone, $\lambda,c>0$. Then the resolvent $J_\lambda(a):= (1+\lambda a)^{-1}: (1+\lambda a)[H]\subseteq  H\to H$ of $a$ is Lipschitz continuous with $\abs{J_\lambda(a)}_{\textnormal{Lip}}\leq  1$. If in addition $a$ is maximal monotone, then $D(J_\lambda(a))=H$. In particular, if $a-c$ is maximal monotone then $a^{-1}: H\to H$ is Lipschitz continuous with  $\abs{a^{-1}}_{\textnormal{Lip}}\leq  \frac1c$.
\end{Sa}

In Section \ref{MT}, in particular in the Theorems \ref{pro:inhom_Dirichlet} and \ref{pro:inhom_Neumann}, we want to deduce from the maximal monotonicity of a relation $a\subseteq H\oplus H$ in the Hilbert space $H$ the respective property for the relation $PaP^*\subseteq U\oplus U$, where $P: H\to U$ denotes the orthogonal projection onto a closed subspace $U\subseteq H$. The question whether a product of the type $BaB^*$, for some continuous $B$, is again maximal monotone is addressed in various publications, cf. e.g. \cite{Asakawa1987,Garcia2006,Pennanen2003,Robinson1999} and the references therein. In particular, in \cite{Robinson1999} conditions are given for the case of real Hilbert spaces. The author of \cite{Robinson1999} uses the theory of convex analysis in his proof. The methods carry over to the complex case. We gather some results concerning maximal monotone relations without proof. 
%

\begin{Sa}[{\cite[Theorem 4]{Robinson1999}}]\label{Thm: PAPMAX} Let $H$ be a Hilbert space, $U\subseteq  H$ a closed subspace and let $a\subseteq  H\oplus H$ be a maximal monotone relation. Moreover, assume that $[H]a=H$. Denote by $P: H\to U$ the orthogonal projection onto $U$. Then the relation $PaP^* \subseteq  U\oplus U$ is maximal monotone.
\end{Sa}
\begin{Fo}{\label{cor:c_max_mon}} Let $H$ be a Hilbert space, $U\subseteq H$ a closed subspace. Denote by $P:H\to U$ the orthogonal projection onto $U$. If $c>0$ and $a\subseteq H\oplus H$ is $c$-maximal monotone with $[H]a=H$, then $PaP^\ast$ is $c$-maximal monotone. 
\end{Fo}
\begin{Le}\label{Le: Inh} Let $H$ be a Hilbert space, $a\subseteq  H\oplus H$ such that $a^{-1}:H\to H$ is Lipschitz-continuous. For $u_0,v_0\in H$ we $a-(u_0,v_0):=\{(x-u_0,y-v_0);(x,y)\in a\}$. Then $a-(u_0,v_0):H\to H$ is Lipschitz-continuous with the same Lipschitz-constant as $a^{-1}$.
\end{Le}

The proof is straight-forward and we omit it.

\begin{rem}
 If $a\subseteq H\oplus H$ is maximal monotone, then $a-(u_0,v_0)$ is also maximal monotone (cf. \cite[Lemma 3.37]{Trostorff2011}).
\end{rem}


%

\section{Solution theory for elliptic boundary value problems}\label{MT}

\subsection{Abstract theorems}\label{AT}

 The first theorem comprises the essential observation of the whole article. It may be regarded as an abstract version of homogeneous boundary value problems for both the Dirichlet and the Neumann case.

\begin{Sa}[solution theory for homogeneous elliptic boundary value problems]\label{Thm: STHE}
Let $H_1,H_2$ be Hilbert spaces and let $A:D(A)\subseteq  H_1\to H_2$ be a densely defined closed linear operator and such that $R(A)\subseteq  H_2$ is closed. Define $B: D(A)\cap N(A)^\bot \to R(A): x\mapsto Ax$ and let $a\subseteq  R(A)\oplus R(A)$ such that $a^{-1}:R(A)\to R(A)$ is Lipschitz-continuous. Then for all $f\in H_{-1}(\abs{B})$ there exists a unique $u\in H_1(\abs{B})$ such that the following inclusion holds
\[
   A^*a A \ni (u, f).
\]
Here $A^\ast$ stands for the continuous extension of $D(A^\ast)\subseteq H_0(|A^\ast|+i)\to H_{-1}(|A|+i):\phi\mapsto A^*\phi$. Moreover, the solution $u$ depends Lipschitz-continuously on the right-hand side with Lipschitz constant $|a^{-1}|_{\mathrm{Lip}}$. \footnote{For a Lipschitz continuous mapping $f: X\to Y$ between two metric spaces $(X,d)$ and $(Y,e)$, we denote by \[\abs{f}_{\textnormal{Lip}}:= \inf\{ c\geq  0; \forall x_1,x_2 \in X: e(f(x_1),f(x_2))\leq  c d(x_1,x_2)\}\] the best Lipschitz constant.}\\
In other words, the relation $(B^* aB)^{-1} \subseteq  H_{-1}(\abs{B}) \oplus  H_1(\abs{B})$ defines a Lipschitz-continuous mapping with $\abs{(B^* aB)^{-1}}_{\textnormal{Lip}}= |a^{-1}|_\mathrm{Lip}$.
\end{Sa}
\begin{proof} It is easy to see that $(u,f)\in  A^*a A$ if and only if $(u,f)\in B^* aB$. Hence, the assertion follows from $(B^* aB)^{-1}=B^{-1}a^{-1}(B^*)^{-1}$, Theorem \ref{Le: unitary} and the fact that $a^{-1}$ is Lipschitz-continuous on $R(A)$.
\end{proof}


\begin{rem}\label{rem: Poincare} \begin{enumerate}[(a)]

\item Theorem \ref{Thm: STHE} especially applies in the case, where $a\subseteq R(A)\oplus R(A)$ is a $c$-maximal monotone relation for some constant $c>0$ by Theorem \ref{Thm: Moro}. The best Lipschitz-constant of $a^{-1}$ can then be estimated by $\frac{1}{c}$.  
 \item  In view of Theorem \ref{Thm: PAPMAX}, there are many maximal monotone relations $a$ such that their respective projections to $R(A)\oplus R(A)$ is maximal monotone. In order to apply Theorem \ref{Thm: STHE} one encounters the difficulty to show that $R(A)\subseteq H_2$ is closed. By the closed graph theorem, the closedness of $R(A)$ is equivalent to the following Poincare-type estimate 
\begin{equation}\label{eq:compact_resolvent}
 \exists c>0\ \forall x\in D(A)\cap N(A)^\bot: |x|_{H_1} \leq c|Ax|_{H_2}.
\end{equation}
A sufficient condition on the operator $A$  to have closed range is that the domain $D(A)$ is compactly embedded into the underlying Hilbert space $H_1$. Indeed, in this case, it is possible to derive an estimate of the form \eqref{eq:compact_resolvent} and therefore our solution theory is applicable.
 \item The latter theorem also gives a possibility to solve the \emph{inverse problem}, i.e., to determine the {}``coefficients'' $a\subseteq R(A)\oplus R(A)$ from the solution mapping ``$f\mapsto u$''. If $a$ is thought to be a $c$-maximal monotone relation in $H_2$ such that $[H_2]a=H_2$ then it is only possible to reconstruct the part $PaP^*$, where $P:H_2 \to R(A)$ denotes the orthogonal projection onto $R(A)$. 
 \end{enumerate}

\end{rem}

Now, we introduce an abstract setting for dealing with inhomogeneous boundary value problems. For this purpose we need a second operator $C$ which is in the Dirichlet-type case an extension and in the Neumann-type case a restriction of our operator $A$. For simplicity we just treat the case where $a\subseteq H_2\oplus H_2$ is $c$-maximal monotone and $[H_2]a=H_2$.

\begin{Sa}[solution theory for inhomogeneous Dirichlet-type problems]\label{pro:inhom_Dirichlet} 
Let $H_1,H_2$ be two Hilbert spaces and $A:D(A)\subseteq H_1 \to H_2$, $C:D(C)\subseteq H_1 \to H_2$ be two densely defined closed linear operators with $A\subseteq C$ and $R(A)\subseteq H_2$ closed. Furthermore, let $a\subseteq H_2\oplus H_2$  be $c$-maximal monotone for some $c>0$ with $[H_2]a=H_2$. Then for each $u_0 \in D(C), f\in H_{-1}(|B|)$ there is a unique $u\in H_1(|C|+i)$ with
\begin{align}{\label{eq:inhom_dirichlet}}
 A^*a C &\ni (u, f) \\
u-u_0 &\in H_1(|B|) \nonumber,
\end{align}
 where $B:D(A)\cap N(A)^\bot \subseteq N(A)^\bot \to R(A)$ is again the restriction of $A$. 
\end{Sa}

\begin{proof} Denote by $P:H_2\to R(A)$ the orthogonal projector onto $R(A)$. We set $\tilde{a}:=  a -(Cu_0,0)$, and obtain again a $c$-maximal monotone relation with $[H_2]\tilde{a}=H_2$. 
We show that $u$ is a solution of (\ref{eq:inhom_dirichlet}) if and only if $u-u_0\in H_1(|B|)$ is the solution of
\begin{equation}{\label{eq:cor_hom_dirichlet}}
 B^* P\tilde{a}P^\ast B  \ni (u-u_0, f).
\end{equation}
Indeed, if $u-u_0$ satisfies this inclusion, then we find $v\in H_2$ such that $(P^\ast B(u-u_0),v)\in \tilde{a}$ and $B^* Pv=f$. By definition of $\tilde{a}$ this implies $(P^\ast B(u-u_0)+Cu_0,v)\in a$ and since $P^\ast=1|_{R(A)}$ we get $(Cu,v)\in a$. This means $u\in H_1(|C|+i)$ solves the problem (\ref{eq:inhom_dirichlet}). If, on the other hand, $u\in H_1(|C|+i)$ satisfies (\ref{eq:inhom_dirichlet}), then we find $v\in H_2$ such that $(Cu,v)\in a$ and $B^* Pv=f$. Since $u-u_0 \in H_1(|B|)$ this implies $(B(u-u_0),v)\in \tilde{a}$ and hence $u-u_0$ solves the problem (\ref{eq:cor_hom_dirichlet}). Since (\ref{eq:cor_hom_dirichlet}) has a unique solution in $H_1(|B|)$ by Theorem \ref{Thm: STHE} and Corollary \ref{cor:c_max_mon}, we get the assertion.
\end{proof}
We may now show a continuity estimate. The proof for this estimate is adopted from \cite[Section 2.5]{Trostorff2011}.
\begin{Fo}[continuity estimate, Dirichlet case]\label{cor: cont_estimate} Let $a,A,C,B$ be as in Theorem \ref{pro:inhom_Dirichlet}. Then there exists $L >0$ such that for all $f,g\in H_{-1}(\abs{B})$, $u_0,v_0 \in D(C)$ with\footnote{Here, for a relation $w\subseteq G_1 \oplus G_2$ for Hilbert spaces $G_1,G_2$ the \emph{adjoint relation} $w^*$ is defined as
\[
    w^* := \{ ( u,-v) \in G_2\oplus G_1; (v,u) \in w \}^{\bot},
\]
where the orthogonal complement is with respect to the scalar product of $G_2\oplus G_1$.
} $C(u_0-v_0)\in [H_2]a^*$ and the respective solutions $u,v \in H_1(\abs{C}+i)$ of 
\[ A^* a C \ni (u,f), u-u_0 \in H_1(|B|) \text{ and } A^*a C\ni (v, g), v-v_0 \in H_1(|B|)\]
 the following estimate holds
\begin{multline*}
 \abs{u-v}_{H_1(\abs{C}+i)} \leq L\Big(\abs{f-g}_{H_{-1}(\abs{B})} + \abs{u_0-v_0}_{H_1(\abs{C}+i)} \\ + \inf\{ \abs{w_0}_{H_2}; (C(u_0-v_0),w_0)\in a^*\}\Big).
\end{multline*}
\end{Fo}
\begin{proof} From the proof of Theorem \ref{pro:inhom_Dirichlet}, we know that $u$ satisfies
\[
 B^* P(a-(Cu_0,0)) P^\ast B\ni (u-u_0, f).
\]
Hence, there exists $x,y \in H_2$ such that
\[
  (P^*B(u-u_0) + Cu_0,x)=(Cu,x)\in a \text{ and }Px = (B^*)^{-1}f
\]
and the respective property for $y$, where $(u_0,f,u)$ is replaced by $(v_0,g,v)$. Let $L_1 >0$ such that for all $h\in H_{1}(\abs{B})$ we have $\abs{h}_{H_1(\abs{C}+i)} \leq L_1 \abs{h}_{H_1(\abs{B})}$.
Then we compute with the help of $P^*B(u-u_0)=Cu-Cu_0$:
\begin{align*}
   \abs{u-v}_{H_{1}(\abs{C}+i)}&\leq \abs{u-u_0-(v-v_0)}_{H_1(\abs{C}+i)}+\abs{u_0-v_0}_{H_1(\abs{C}+i)} \\
                               &\leq L_1 \abs{u-u_0-(v-v_0)}_{H_1(\abs{B})}+\abs{u_0-v_0}_{H_1(\abs{C}+i)}\\
                               & = L_1 \abs{B(u-u_0)-B(v-v_0)}_{H_0(\abs{B^*})}+\abs{u_0-v_0}_{H_1(\abs{C}+i)}\\
                               & = L_1 \abs{P^*B(u-u_0)-P^*B(v-v_0)}_{H_2} +  \abs{u_0-v_0}_{H_1(\abs{C}+i)}\\
                               & = L_1 \abs{Cu-Cv}_{H_2}+L_1\abs{Cu_0-Cv_0}_{H_2} +  \abs{u_0-v_0}_{H_1(\abs{C}+i)}.
\end{align*}
Thus, it suffices to estimate $\abs{Cu-Cv}_{H_2}$. To this end, let $w_0 \in H_2$ be such that $(C(v_0-u_0), w_0)\in a^*$. Using the monotonicity of $a-c$ and the definition of $a^*$, we conclude that
\begin{align*}
  &\Re \langle (B^*)^{-1}f-(B^*)^{-1}g, B(u-u_0)-B(v-v_0)\rangle \\
  &=\Re \langle Px-Py,   B(u-u_0)-B(v-v_0)\rangle \\
  &=\Re \langle x-y,   P^*B(u-u_0)-P^*B(v-v_0)\rangle \\
  &=\Re \langle x-y,   Cu-Cv\rangle+ \Re\langle x-y, Cv_0 - Cu_0\rangle \\
  &\geq c \abs{Cu-Cv}_{H_2}^2+ \Re \langle Cu-Cv, w_0\rangle \\
  &\geq c \abs{Cu-Cv}_{H_2}^2 - \abs{Cu-Cv}_{H_2}\abs{w_0}_{H_2}.
\end{align*}
Applying the Cauchy-Schwarz-inequality to the left-hand side, we get for $\eps >0$
\begin{align*}
 &c\abs{Cu-Cv}_{H_2}^2 \\
 &\leq \abs{f-g}_{H_{-1}(\abs{B})}\abs{B(u-u_0)-B(v-v_0)}_{H_0(\abs{B^*})}+\abs{w_0}_{H_2}\abs{Cu-Cv}_{H_2} \\
 &\leq \abs{f-g}_{H_{-1}(\abs{B})}\abs{Cu_0-Cv_0}_{H_2} \\ &\quad +\abs{f-g}_{H_{-1}(\abs{B})}\abs{Cu-Cv}_{H_2}+\abs{w_0}_{H_2}\abs{Cu-Cv}_{H_2}\\
 &\leq \abs{f-g}_{H_{-1}(\abs{B})}\abs{Cu_0-Cv_0}_{H_2}+\frac{1}{2\eps}(\abs{f-g}_{H_{-1}(\abs{B})}+\abs{w_0}_{H_2})^2+ \frac{\eps}{2}\abs{Cu-Cv}_{H_2}^2. 
\end{align*}
For $\eps >0$ small enough, this yields an estimate for $\abs{Cu-Cv}_{H_2}$ in terms of $\abs{w_0}_{H_2}$, $\abs{Cu_0-Cv_0}$ and $\abs{f-g}_{H_{-1}(\abs{B})}$. 
%
\end{proof}
\begin{rem}
 The norm in the above corollary can be interpreted as the ``graph-norm'' of $a^*$. We also shall briefly discuss two extreme cases of the above corollary. Since $a^*$ is a linear relation, $0 \in [H]a^*$. Thus, we have a continuous dependence result for varying right-hand sides and fixed boundary data. If $a$ is a bounded linear mapping, then $[H_2]a^*=H_2$. Therefore the condition $C(u_0-v_0)\in [H_2]a^*$ is trivially satisfied and the term $\inf\{ \abs{w_0}_{H_2}; (C(u_0-v_0),w_0)\in a^*\}$ can be estimated by $\Abs{a}\abs{C(u_0-v_0)}_{H_2}$, where $\Abs{a}$ is the operator norm of $a : H_2\to H_2$.
\end{rem}

\begin{Sa}[solution theory for inhomogeneous Neumann-type problems]\label{pro:inhom_Neumann}
Let $H_1,H_2$ be two Hilbert spaces and $A:D(A)\subseteq H_1 \to H_2,C:D(C)\subseteq H_1 \to H_2$ be two densely defined closed linear operators with $C\subseteq A$ and $R(A)$ closed in $H_2$. Furthermore, let $a\subseteq H_2\oplus H_2$  be $c$-maximal monotone for some $c>0$ with $[H_2]a=H_2$. 
Then for each $f\in H_{-1}(|C|+i),u_0\in H_2$ with $f-C^\ast u_0 \in H_{-1}(|B|)$\footnote{This means that we find an element $\xi\in H_{-1}(|B|)$ such that
\begin{equation*}
 (f-C^\ast u_0)(w)=\xi(w)\quad (w\in H_1(|B|)\cap H_1(|C|+i))
\end{equation*}
in the sense of Remark \ref{rem:ExDua}\eqref{rem:duality}} there exists a unique $u\in H_1(|B|)$ such that
\begin{equation}{\label{eq:inhom_neumann}}
 C^\ast aA  \ni (u, f), 
\end{equation}
in the sense that we find $v\in a[\{Au\}]$ such that (cp. Remark \ref{rem:ExDua}\eqref{rem:duality})
\begin{equation*}
f(w) = \langle v ,Cw \rangle_{H_2}=(C^\ast v) (w)\quad (w\in H_1(|B|)\cap H_1(|C|+i))
\end{equation*} and
\begin{equation}\label{eq:bdy_cond}
 A^*(v-u_0) (w) = 0 \quad \left(w\in (H_1(\abs{B})\cap H_1(\abs{C}+i))^{\bot_{H_1(\abs{B})}}\right).
\end{equation}
\end{Sa}

\begin{proof} 
 Consider the following problem of finding $u\in H_1(\abs{B})$ such that
\begin{equation}{\label{eq:cor_hom_neumann}}
B^\ast P \tilde{a} P^\ast B \ni (u,\xi), 
\end{equation}
holds, where $\tilde{a}:=  a -(0,u_0)$ and $\xi\in H_{-1}(|B|)$ satisfies $\xi|_{H_1(|C|+i)}=(f-C^\ast u_0)|_{H_1(|B|)}$ with $\xi = 0$ on $(H_1(|B|)\cap H_1(|C|+i))^{\bot_{H_1(\abs{B})}}$.  Note that such a choice for $\xi$ is possible, since $H_1(\abs{B})\cap H_1(\abs{C}+i)\subseteq H_1(\abs{B})$ is closed. Indeed, $B$ and $C$ are both closed linear operators restricting $A$. Hence, $H_{1}(\abs{B})$ and $H_1(\abs{C}+i)$ are closed subspaces of $H_1(\abs{A}+i)$. Thus, the norms of the spaces $H_{1}(\abs{B})$ and $H_1(\abs{A}+i)$ are equivalent on $H_1(\abs{B})$ and therefore $H_1(\abs{B})\cap H_1(\abs{C}+i)\subseteq H_1(\abs{B})$ is closed.

We show that the problem (\ref{eq:cor_hom_neumann}) is equivalent to \eqref{eq:inhom_neumann}. Then the assertion follows from Theorem \ref{Thm: STHE} and Corollary \ref{cor:c_max_mon}. So let $u\in H_1(|B|)$ be a solution of (\ref{eq:cor_hom_neumann}). That means that we find $y\in H_2$ such that $(Bu,y)\in \tilde{a}$ and $B^\ast Py=\xi$. This, however, implies $(Bu,y+u_0)\in a$ and for $w\in H_1(|B|)\cap H_1(|C|+i)$ we compute
\begin{align*}
 \langle y+u_0 , Cw\rangle_{H_2} &= \langle y,Cw\rangle_{H_2}+(C^\ast u_0)(w) \\
&= \langle Py, Bw\rangle_{H_2}+(C^\ast u_0)(w) \\
&= (B^\ast Py)(w)+(C^\ast u_0)(w) \\
&= (f-C^\ast u_0)(w)+(C^\ast u_0)(w) \\
&= f(w) =\langle f,w\rangle_{H_1}.
\end{align*}
Moreover, for $w\in (H_1(\abs{B})\cap H_1(\abs{C}+i))^{\bot_{H_1(\abs{B})}}$, we have $y+u_0 \in a[\{Bu\}]$ and $B^*P(y+u_0-u_0)(w) = B^*Pv (w)=\xi(w)=0$.  Thus, $u$ is a solution of (\ref{eq:inhom_neumann}) in the stated sense. If, on the other hand, $u\in H_1(|B|)$ solves problem (\ref{eq:inhom_neumann}), then we find $v\in H_2$ with $(Bu,v)\in a$ and $(C^\ast v)|_{H_1(|B|)}=f|_{H_1(|B|)}$. It suffices to show that $f-C^\ast u_0$ and $B^\ast P(v-u_0)$ coincide on $H_1(|B|)\cap H_1(|C|+i)$. For this purpose let $w\in H_1(|C|+i)\cap H_1(|B|)$. Then we compute
\begin{align*}
(B^\ast P(v-u_0))(w) &= \langle P(v-u_0), Bw \rangle_{H_2} \\
&= \langle v,Bw\rangle_{H_2}-\langle u_0,Bw\rangle_{H_2} \\
&= \langle v,Cw\rangle_{H_2}-\langle u_0, Cw\rangle \\
&= (C^\ast v)(w)-(C^\ast u_0)(w) \\
&= (f-C^\ast u_0)(w).
\end{align*}
Hence, by the definition of $\tilde{a}$ we derive that $u$ solves (\ref{eq:cor_hom_neumann}) with $\xi=B^\ast P(v-u_0)\in H_{-1}(|B|)$.
\end{proof}
\begin{rem}\label{rem:weak_solution} 
  \begin{enumerate}[(a)] 
   \item The solvability condition may seem a bit awkward, but it is largely unavoidable if one wants to maintain uniqueness of the solution by showing the equivalence of the inclusions \eqref{eq:inhom_neumann} and \eqref{eq:cor_hom_neumann}. The very reason for this formulation is the fact that the spaces of linear functionals $H_{-1}(\abs{B})$ and $H_{-1}(\abs{C}+i)$ cannot be compared. However, one may also interpret condition \eqref{eq:bdy_cond} as the realization of the boundary condition in a distributional sense.
   \item It should be noted that the very weak formulation, how the inclusion (\ref{eq:inhom_neumann}) holds, may lead to unexpected solutions. Let for instance $f\in H_1,u_0=0$ and let the relation $a$ be given by $a=\mathrm{id}_{H_2}$. Then $f$ is in $H_{-1}(|B|)$ in the sense of Theorem \ref{pro:inhom_Neumann}. This is due to the Riesz representation theorem, since 
\begin{equation*}
  H_1(|B|) \ni v \mapsto \langle f,v\rangle_{H_1}
\end{equation*}
 defines a linear continuous functional on $H_1(|B|)$. Thus, we find $\eta \in H_1(|B|)$ such that
\begin{equation*}
 \langle |B|\eta , |B|v\rangle_{H_1}=\langle f,v\rangle_{H_1}
\end{equation*}
for every $v\in H_{1}(|B|)$. Hence, $\xi:= B^\ast B \eta$ satisfies $\xi=f|_{H_1(|B|)}$. So, according to Theorem \ref{pro:inhom_Neumann} we find a unique $u\in H_1(|B|)$ such that $C^\ast Bu|_{H_1(|B|)}=f|_{H_1(|B|)\cap H_1(|C|+i)}$ and $A^*Bu(w) = 0$ for all $w\in (H_1(\abs{B})\cap H_1(\abs{C}+i))^{\bot_{H_1(\abs{B})}}$. If $f\in N(A)$ then we get 
\begin{equation*}
 \forall v\in H_1(|C|+i)\cap H_1(|B|): \langle Bu, Cv\rangle_{H_2} =0.
\end{equation*}
Since we already know that the solution $u$ is unique, we conclude $u=0$. Thus $u=0$ solves the problem (\ref{eq:inhom_neumann}) for any right-hand side $f\in N(A)$. Since in applications this is not desirable one usually assumes $f \in N(A)^\bot$.
\end{enumerate}
\end{rem}
We also have a continuous dependency result.
\begin{Fo}[continuity estimate, Neumann case]\label{cor: cont_est_vn} Let $a,A,C,B$ be as in Theorem \ref{pro:inhom_Neumann}. Then there is $L >0$ such that for all $f,g\in H_{-1}(\abs{B})$, $u_0,v_0 \in H_2$ with $f-C^*u_0, g-C^*v_0 \in H_{-1}(\abs{B})$ and the respective solutions $u,v \in H_1(\abs{B})$ of 
\[ C^* a A\ni (u, f) \text{ and } C^* a A \ni(v, g) \]
 the following estimate holds
\begin{multline*}
 \abs{u-v}_{H_1(\abs{B})}   \leq \frac{1}{c}\sup\Big\{ \abs{((f-C^\ast u_0)-(g-C^\ast v_0))(w)}; \\ w\in H_1(\abs{B})\cap H_1(\abs{C}+i), \abs{w}_{H_1(\abs{B})}=1\Big\} + \frac{1}{c}\abs{Pu_0-Pv_0}_{H_0(\abs{B^*})}.
\end{multline*}
\end{Fo}
\begin{proof}
Let $\xi,\eta \in H_{-1}(\abs{B})$ be such that $\xi|_{H_1(|C|+i)}=(f-C^\ast u_0)|_{H_1(|B|)}$ and $\eta|_{H_1(|C|+i)}=(g-C^\ast v_0)|_{H_1(|B|)}$ and $\xi=\eta=0$ on $(H_1(\abs{B})\cap H_1(\abs{C}+i))^{\bot_{H_1(\abs{B})}}$. Observe that \eqref{eq:cor_hom_neumann} is the same as to say
\[
  P(a -(0,u_0)) P^*  \ni (Bu,(B^*)^{-1}\xi).
\]
Hence, we get
\[
  PaP^* \ni (Bu,(B^*)^{-1}\xi+Pu_0).
\]
Invoking Theorems \ref{Thm: PAPMAX}, \ref{Thm: Moro} and \ref{Le: unitary}, we conclude that
\begin{align*}
   \abs{u-v}_{H_1(\abs{B})} & = \abs{Bu-Bv)}_{H_{0}(\abs{B^*})} \\
                            &\leq \frac{1}{c} \abs{(B^*)^{-1}\xi+Pu_0-(B^*)^{-1}\eta+Pv_0)}_{H_{0}(\abs{B^*})}\\
                            &\leq \frac{1}{c} \abs{\xi-\eta}_{H_{-1}(\abs{B})}+\frac{1}{c}\abs{Pu_0-Pv_0}_{H_{0}(\abs{B^*})}.\qedhere
\end{align*}
\end{proof}

\subsection{Examples}\label{ex}

In order to apply the results of Section \ref{AT} to concrete cases, we have to maintain the assumptions made on the abstract operator $A$, i.e., mainly, it is important to obtain the closedness of the range of $A$. We study examples, when this can be ensured. For the following let $n\in\N$ and let $\Omega\subseteq  \R^n$ be an open subset.

\subsubsection{Potential Theory}

\begin{Def} We define
\begin{align*}
  \widetilde{\diverg}_c :& \, C_c^\infty(\Omega)^n\subseteq  \bigoplus_{k=1}^nL_2(\Omega) \to L_2(\Omega) \\
                                            & \phi=(\phi_1,\ldots,\phi_n)^T\mapsto \sum_{k=1}^n \partial_k\phi_k,
\end{align*}
where $\partial_k$ denotes the derivative with respect to the $k$'th
variable ($k\in\{1,\ldots,n\}$). Furthermore, define
\begin{align*}
  \widetilde{\grad}_c :& \, C_c^\infty(\Omega)\subseteq  L_2(\Omega) \to \bigoplus_{k=1}^n L_2(\Omega) \\
                                            & \phi\mapsto (\partial_1\phi,\ldots,\partial_n\phi)^T.
\end{align*}
Moreover, let $\diverg:= -\left(\widetilde{\grad}_c\right)^*$, $\grad:= -\left(\widetilde{\diverg}_c\right)^*$, $\diverg_c:= -\grad^*$ and $\grad_c:= - \diverg^*$. 
\end{Def}

We like to state some examples, how the theory developed in Section \ref{AT} can be used to obtain a solution theory for inhomogeneous Dirichlet and Neumann type problems for the Laplacian. It should be noted that the theory does not require any regularity for the boundary of $\Omega$. Instead,  we assume that the boundary data is given as a function on the whole domain $\Omega$.\newline

\label{Ex: POTD} For the Dirichlet case, assume additionally that $\Omega$ is bounded in one dimension. Let $a\subseteq  L_2(\Omega)^n \oplus L_2(\Omega)^n$ be a $c$-maximal monotone relation for some $c>0$ such that $[L_2(\Omega)^n]a=L_2(\Omega)^n$. For every $f \in H_{-1}(\abs{\grad_c})$ and $u_0\in D(\grad)$ there is a unique solution $u \in H_{1}(\abs{\grad}+i)$ such that the inclusions
\begin{align*}
    \diverg a \grad  &\ni (u, f) \\
     u-u_0 &\in D(\grad_c)
\end{align*}
 are satisfied. Moreover, the solution $u$ depends Lipschitz-continuously on $f$ and $u_0$ in the sense of Corollary \ref{cor: cont_estimate}. Indeed, by our general reasoning in Theorem \ref{Thm: STHE}, it suffices to show the closedness of $R(\grad_c)$. The latter follows from the Poincare-inequality (cf. \cite[Satz 7.6, p.120]{Wloka1982}), cf. also Remark \ref{rem: Poincare}

 \[
     \Abs{ u }_{L_2(\Omega)} \leq  C \Abs{\grad_c u}\qquad (u \in H_1(\abs{\grad_c}))
 \]
 for some suitable constant $C>0$.
\newline
 
For the Neumann case, assume additionally that $\Omega$ is bounded, connected and satisfies the segment property. According to Rellich's theorem (cf. \cite[Theorem 3.8, p.24]{agmon2010lectures})  we obtain
\[
 \Abs{u}_{L_2(\Omega)}\leq C \Abs{\grad u} 
\]
 for all $u\in D(\grad)\cap N(\grad)^\bot$. Since $\Omega$ is connected, the null space of $\grad$ is given by the constant functions, i.e. $N(\grad)=\Lin\{1\}$. According to Remark \ref{rem: Poincare}  our solution theory applies and thus for every $f \in H_{-1}(|\grad_c|+i)$ and $u_0\in L_2(\Omega)^n$, satisfying
\begin{equation} \label{eq:compatibility_cond}
 f-\diverg u_0 \in H_{-1}(\abs{\grad|_{\{1\}^\bot}})
\end{equation}
in the sense of Theorem \ref{pro:inhom_Neumann}, we get the unique existence of $u \in H_{1}(\abs{\grad|_{\{1\}^\bot}})$ such that the inclusion
 \[
    \diverg a \grad  \ni (u, f)
 \]
 holds.

\begin{rem}
 In \cite[Theorem 4.22, p.78]{CioDon} we find for $f\in L_2(\Omega),u_0\in D(\diverg)$ a compatibility condition of the form 
\begin{equation*}
 \langle f-\diverg u_0,1\rangle =0.
\end{equation*}
In our framework, this is just the assumption to avoid contra-intuitive solutions $u$ (cp. Remark \ref{rem:weak_solution}).
\end{rem}

\subsubsection{Stationary Elasticity}

 We only consider in more detail the homogeneous Neumann-type problem, and refer to the abstract solution theory for inhomogeneous Dirichlet-type problems.

\begin{Def} Let $H_{\textnormal{sym}}(\Omega)$ be the vector space of $L_2(\Omega)$-valued symmetric $n\times n$ matrices, i.e.
 \[
    H_{\textnormal{sym}}(\Omega):= \{ \Phi\in L_2(\Omega)^{3\times 3}; \text{ for a.e.\ } x\in \Omega: \Phi(x)^T=\Phi(x)\}.
 \]
 Endowing $H_{\textnormal{sym}}(\Omega)$ with the inner product
 \[
    H_{\textnormal{sym}}(\Omega)\times H_{\textnormal{sym}}(\Omega) \ni (\Phi,\Psi)\mapsto \int_{\Omega} \trace(\Phi(x)^*\Psi(x))\dd x
 \]
 the space $H_{\textnormal{sym}}(\Omega)$ becomes a Hilbert space. 
 \end{Def}

\begin{Def} With
 \begin{align*}
    \widetilde{\Diverg_c} :& \,  H_{\textnormal{sym}}(\Omega)\cap C_c^\infty(\Omega)^{n\times n}\subseteq  H_{\textnormal{sym}}(\Omega)\to L_2(\Omega)^n\\
    											& (T_{jk})_{(j,k)\in \{1,\ldots,n\}^2}\mapsto \left(\sum_{k=1}^n\partial_kT_{jk}\right)_{j\in\{1,\ldots,n\}}
 \end{align*}
 and
 \begin{align*}
    \widetilde{\Grad_c} :& \, C_c^\infty(\Omega)^n\subseteq  L_2(\Omega)^{n}\to H_{\textnormal{sym}}(\Omega) \\
    										& (\Phi_k)_{k\in\{1,\ldots,n\}}\mapsto \frac12\left((\partial_k\Phi_j)_{(j,k)\in\{1,\ldots,n\}^2}+(\partial_j\Phi_k)_{(j,k)\in\{1,\ldots,n\}^2}\right),
  \end{align*}
  we define $\Diverg:= -\left(\widetilde\Grad_c\right)^*, \Grad:= - \left(\widetilde\Diverg_c\right)^*, \Diverg_c:= -\Grad^*$ and $\Grad_c:= -\Diverg^*$.
 \end{Def}

In the view of Remark \ref{rem: Poincare} we want to establish $\Grad$ in domains $\Omega$ such that $\Grad$ has a compact resolvent. In \cite{Weck1994} these domains were treated as domains having the elastic compactness property. Classically $\Omega$ is assumed to be bounded and satisfies the cone-property in order to apply Korn's inequality and the Poincare inequality. However in \cite{Weck1994} it was shown that these condition could be relaxed. So for instance, the domain $\Omega$ is allowed to have cusps of certain types (cf. \cite[Theorem 2]{Weck1994}), where Korn's inequality is not applicable. So let us assume that $\Omega$ has the elastic compactness property. Furthermore let $a\subseteq  H_{\textnormal{sym}}(\Omega)\oplus H_{\textnormal{sym}}(\Omega)$ be $c$-maximal monotone for some constant $c>0$ and assume that $[H_{\textnormal{sym}}(\Omega)]a=H_{\textnormal{sym}}(\Omega)$. Then for every $f\in H_{-1}(\abs{\Grad|_{N(\Grad)^\bot}})$ there exists a unique $u\in H_1(\abs{\Grad|_{N(\Grad)^\bot}})$ such that the inclusion
 \[
    \Diverg_c  a \Grad  \ni (u,f)
 \]
 holds.

\subsubsection{Electro- and Magneto-statics}

Our last example considers elliptic problems where the operator $A$ is given by $\curl$. These types of equations can be found in the field of electro- and magneto-statics (cf. \cite{MilaniPicard1988}). We restrict ourselves to the case of $n=3$.

\begin{Def} We define
\begin{align*}
   \widetilde\curl_c :& \, C_c^\infty(\Omega)^3 \subseteq  \bigoplus_{k=1}^3L_2(\Omega)\to \bigoplus_{k=1}^3L_2(\Omega) \\
                      & \phi=\begin{pmatrix}\phi_1\\ \phi_2\\ \phi_3\end{pmatrix} \mapsto \begin{pmatrix} \partial_2\phi_3-\partial_3\phi_2 \\ \partial_3\phi_1-\partial_1\phi_3 \\ \partial_1\phi_2-\partial_2\phi_1
                      \end{pmatrix}.
\end{align*} Define $\curl := \left(\widetilde\curl_c\right)^*$ and $\curl_c:=\curl^*$.
\end{Def}

We want to establish the operator $\curl$ in a suitable setting, such that $D(\curl)\hookrightarrow \hookrightarrow L_2(\Omega)^3$. This problem was studied for instance in \cite{Picard1984,Witsch1993,Picard2001}. In \cite{Witsch1993} it was shown that for bounded domains $\Omega$ satisfying the segment property and $\mathbb{R}^3\setminus \overline{\Omega}$ having the $p$-cusp-property for $p<2$ (cf. \cite[Definition 3]{Witsch1993}) the embeddings
\begin{equation*}
 H_1(|\curl|+i)\cap H_1(|\diverg_c|+i) \hookrightarrow L_2(\Omega)^3, \quad H_1(|\curl_c|+i)\cap H_1(|\diverg|+i) \hookrightarrow L_2(\Omega)^3
\end{equation*}
are compact. Following \cite{MilaniPicard1988}, we can decompose $H_1(|\curl_c|+i)$ and $H_1(|\curl|+i)$ in the following way
\begin{align*}
 H_1(|\curl|+i)&=\overline{\grad[H_1(|\grad|+i)]}\oplus (H_1(|\curl|+i)\cap N(\diverg_c)) \\
H_1(|\curl_c|+i)&= \overline{\grad_c[H_1(|\grad_c|+i)]}\oplus (H_1(|\curl_c|+i)\cap N(\diverg)). 
\end{align*}
Combining these two results, we obtain that
\begin{align*}
 \curl_{c,\sigma} :&\,  D(\curl_c)\cap N(\diverg) \subseteq N(\diverg) \to L_2(\Omega)^3 \\
 \curl_\sigma :& \, D(\curl)\cap N(\diverg_c) \subseteq N(\diverg_c) \to L_2(\Omega)^3 
\end{align*}
are densely defined closed operators with compactly embedded domains. Thus, by Remark \ref{rem: Poincare}, the problems
\begin{equation*}
 (\curl_{c,\sigma}|_{N(\curl_{c,\sigma})^\bot})^\ast a \curl_{c,\sigma}|_{N(\curl_{c,\sigma})^\bot}  \ni (u, f) 
\end{equation*}
and
\begin{equation*}
(\curl_\sigma|_{N(\curl_\sigma)^\bot})^\ast a \curl_\sigma|_{N(\curl_\sigma)^\bot} \ni (u,f)
\end{equation*}
are well-posed in the sense of Theorem \ref{Thm: STHE}. Here again, we assume that $a\subseteq L_2(\Omega)^3\oplus L_2(\Omega)^3$ is a $c$-maximal monotone relation for some $c>0$ with $[L_2(\Omega)^3]a=L_2(\Omega)^3$.

 
%

\bibliographystyle{plain}

%

\end{document}